\documentclass[11pt,reqno]{amsart}
\usepackage[activeacute,english]{babel}
\pdfoutput=1
\usepackage{amssymb}
\usepackage{amsmath,latexsym,amsthm}
\usepackage{colonequals}
\usepackage{multicol}

\usepackage{graphicx}

\newcommand{\eq}[1]{\begin{equation}#1\end{equation}}
\newcommand{\equ}[1]{\begin{equation*}#1\end{equation*}}

\newtheorem{teo}{Theorem}[section]
\newtheorem{lema}{Lemma}[section]
\newtheorem{prop}{Proposition}[section]

\newcommand{\foral}{\hbox{ for all } }
\newcommand{\ve}{\varepsilon}
\newcommand{\la}{\lambda }

\makeatletter \renewenvironment{proof}[1][\proofname] {\par\pushQED{\qed}\normalfont\topsep6\p@\@plus6\p@\relax\trivlist\item[\hskip\labelsep\bfseries#1\@addpunct{.}]\ignorespaces}{\popQED\endtrivlist\@endpefalse} \makeatother

\numberwithin{equation}{section}
\def\R{\mathbb{R}}
\def\N{\mathbb{N}}

\author{Manuel del Pino}
\address{\noindent Manuel del Pino -- Departamento de Ingenier\'{\i}a Matem\'atica and
CMM, Universidad de Chile, Casilla 170 Correo 3, Santiago, Chile.}
\email{delpino@dim.uchile.cl}

\author{Carlos Rom\'an}
\address{\noindent Carlos Rom\'an -- Departamento de Ingenier\'{\i}a Matem\'atica, Universidad de Chile, Santiago, Chile. Universit\'e Pierre et Marie Curie, Paris, France.}
\email{croman@dim.uchile.cl}

\title[Conformal metrics with prescribed Gauss curvature]{\textbf{Large Conformal metrics with prescribed sign-changing Gauss curvature}}
\begin{document}
\maketitle
\begin{abstract}
Let $(M,g)$ be a two dimensional compact Riemannian manifold of genus $g(M)>1$. Let $f$ be a smooth function on $M$ such that
$$f \ge  0, \quad f\not\equiv 0, \quad  \min_M f = 0. $$
Let $p_1,\ldots,p_n$  be any set of points at which $f(p_i)=0$ and $D^2f(p_i)$ is non-singular.
We prove that for all sufficiently small $\lambda>0$
 there exists a family of ``bubbling" conformal metrics $g_\lambda=e^{u_\lambda}g$ such that their Gauss curvature is given by the sign-changing function $K_{g_\lambda}=-f+\lambda^2$. Moreover, the family $u_\lambda$ satisfies
 $$
 u_\lambda(p_j) = -4\log\la -2\log \left (\frac 1{\sqrt{2}} \log \frac 1\la \right ) +O(1)
 $$
 and
  $$\lambda^2e^{u_\lambda}\rightharpoonup8\pi\sum_{i=1}^{n}\delta_{p_i},\quad \mbox{as }\lambda \to 0,$$
where $\delta_{p}$ designates Dirac mass at the point $p$.
\end{abstract}

\section{Introduction}
Let $(M,g)$ be a two-dimensional compact Riemannian manifold. We consider in this paper the classical {\em prescribed Gaussian curvature problem}: Given a real-valued, sufficiently smooth funtion $\kappa(x) $ defined on $M$, we want to know if $\kappa $ can be realized as the
Gaussian curvature $K_{g_1}$ of $M$ for a metric $g_1$, which is in addition conformal to $g$, namely  $g_1=e^{u}g$ for some scalar function $u$ on $M$.

 It is well known, by the uniformization theorem,  that without loss of generality we may assume that $M$ has constant Gaussian curvature for $g$, $K_g=:-\alpha$. Besides, the relation
 $K_{g_1}= \kappa $
 is  equivalent to the following nonlinear partial differential equation
\eq{\Delta_g u+ \kappa\, e^u +\alpha=0,\quad \mbox{in }M,\label{KW}}
where $\Delta_g$ is the Laplace Beltrami operator on $M$.
There is a considerable literature on necessary and sufficient conditions on the function $\kappa $ for the solvability of the PDE (\ref{KW}). We refer the reader in particular to
the classical references \cite{berger,cgy,kw1,kw2,kw3,moser} and to \cite{struwe} for a recent review of the state of the art for this problem.

Integrating equation (\ref{KW}), assuming that $M$ has surface area equal to one,  and using the Gauss-Bonet formula we obtain
\eq{\int_M  \kappa e^ud\mu_g =\int_M K_g d\mu_g=   -\alpha  =  2\pi \chi(M),\label{KW1}}
where $\chi(M)$ is the Euler characteristic of the manifold $M$.

\medskip
In what follows we shall assume that the surface $M$ has {\bf genus $g(M)$ greater than one}, so that $\chi(M)=2(1-g(M))<0$ and hence  $$-K_g = \alpha >0. $$
Then (\ref{KW1}) tells us that a necessary condition for existence is that $\kappa(x) $ be negative somewhere on $M$. More than this, we must have that
$$\int_M \kappa  d\mu_g<0.$$ Indeed testing equation (\ref{KW})  against $e^{-u}$ we get
\eq{\int_M \kappa d\mu_g=-\int_M (|\nabla_g u|^2 +\alpha)e^{-u}d\mu_g< 0. \label{KW2}}

\medskip
Solutions $u$ to equation (\ref{KW1}) correspond to critical points in the Sobolev space
$H^1(M,g)$ of the energy functional
\equ{E_\kappa (u)=\frac12\int_M |\nabla_g u|^2 d\mu_g\,-\, \alpha\int_M ud\mu_g-\int_M \kappa e^{u}d\mu_g.}
As observed in \cite{berger}, since  $\alpha >0$, we have that
 If $\kappa \leq 0$ and $\kappa \not\equiv 0$, then this functional is strictly convex and coercive in $H^1(M,g)$. It thus have a unique critical point which is a global minimizer of $E_\kappa $.

A natural question to ask is what happens when $f$ changes sign. A drastic change in fact occurs. If
$\sup_M \kappa >0$, then the functional $E_\kappa$ is no longer bounded below, hence a global minimizer cannot exist. On the other hand, intuition would tell us that if $\kappa $ is ``not too positive" on a set ``not too big", then the global minimizer should persist in the form of a local minimizer. This is in fact true, and quantitative forms of this statement can be found in
\cite{ab,b}.

\medskip
We shall focus in what follows in a special class of functions $\kappa(x) $ which change sign being nearly everywhere negative.
Let $f$ be a function of class $C^3(M)$ such that $$f \ge  0, \quad f\not\equiv 0, \quad  \min_M f = 0. $$
For $\lambda>0$ we let
$$
{ \kappa_\lambda (x)= -f(x) + \lambda^2.     }
$$
so that our problem now reads
\eq{\Delta_g u -fe^u+\lambda^2e^u+\alpha=0,\quad \mbox{in }M. \label{EQ}}
In \cite{g}, Ding and Liu proved that the global minimizer of $E_{\kappa_0}$ persists as a local minimizer  $\underline{u}_\lambda $ of $E_{\kappa_\lambda}$ for any $0<\lambda <\lambda_0$. From (\ref{KW2}) we see that
$$\lambda_0 < \left(\int_M f\right)^{1/2}.$$
Moreover, they  established the existence of a second, non-minimizing solution $u_\la$ in this range.
Uniqueness of the solution $u_0$ for $\la=0$, and its minimizing character, tell us that we must have $\underline{u}_\la \to u_0$ as $\la\to 0$ while $u_\la$ must become unbounded. The situation is depicted as a bifurcation diagram in Figure \ref{fig1}.

\begin{figure}
	\includegraphics[scale=1.1]{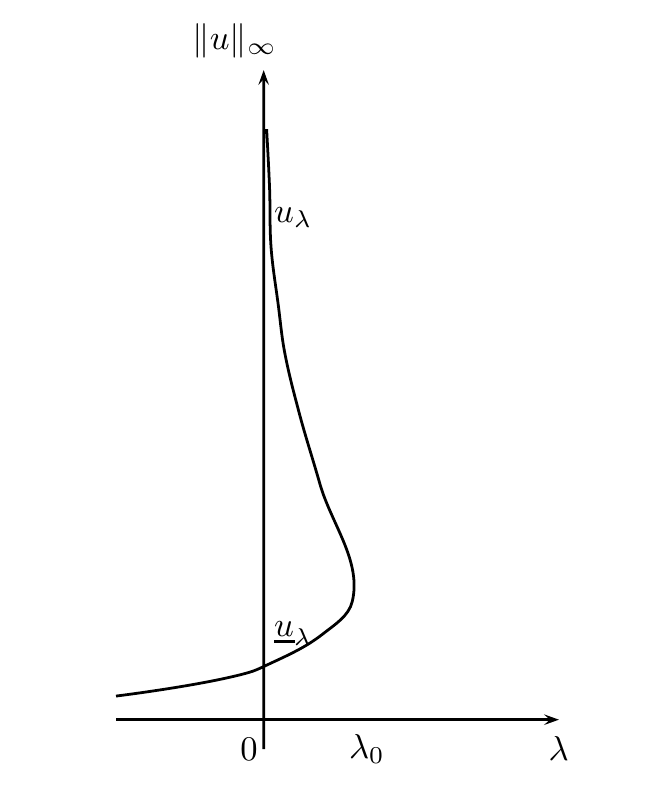}
\caption{Bifurcation diagram  for solutions of Problem \eqref{EQ}}
\label{fig1}\end{figure}

\medskip
The proof in \cite{g} does not provide information on its asymptotic blowing-up behavior or about the number of such ``large" solutions. 
Borer, Galimberti and Struwe \cite{struwe} have recently provided a new construction of the mountain pass solution for small $\la$, which allowed them to identify further properties of it under the following generic assumption: points of global minima of $f$ are non-degenerate. This means that if $f(p)=0$ then $D^2f(p)$ is positive definite.
In \cite{struwe} it is established that blowing-up of the family of large solutions $u_\la$ occurs only near zeros of $f$, and the associated metric exhibits ``bubbling behavior", namely  Euclidean spheres emerge around some of the zero-points of $f$. In fact, the mountain-pass characterization let them estimate the number of bubbling points as no larger than four.
More precisely, they find that along any sequence $\la=\la_k \to 0$, there exist points $p_1^k,\ldots, p_n^k$, $1\leq n \leq 4$, converging to $p_1,\ldots,p_n$ points of global minima of $f$ such that one of the following holds
\begin{enumerate}
\item[(i)] There exist $\varepsilon_\la^1,\ldots,\varepsilon_\la^k$, such that $\varepsilon_\la^i/\la \to 0$, $i=1,\ldots,k$, and in local conformal coordinates around $p_i$ there holds
\eq{u_\la(\varepsilon_\la^ix)-u_\la(0)+\log 8\rightarrow  w(x):=\log\frac{8}{(1+|x|^2)^2},\label{limite}}
smoothly locally in $\R^2$. We note that
\equ{\Delta w + e^w=0.}
\item[(ii)] In local conformal coordinates around $p_i$, with a constant $c_i$ there holds
\equ{u_n(\la x)+4\log(\la)+c_i\rightarrow w_\infty(x),}
smoothly locally in $\R^2$, where $w_\infty$ satisfies
\equ{\Delta_g w_\infty +[1-(Ax,x)]e^{w_\infty}+\alpha=0}
where $A=\frac{1}{2}D^2f(p_i)$.
\end{enumerate}
In this paper we will substantially clarify the structure of the set of large solutions of problem (\ref{EQ}) with a method that yields both multiplicity and accurate estimates of their blowing-up behavior. Roughly speaking we establish that for any given collection of non-degenerate global minima of $f$, $p_1,\ldots,p_k$, there exist a solution $u_\la$ blowing-up in the form (\ref{limite}) {\em exactly} at those points. Moreover \equ{\varepsilon_\lambda^i\sim \frac{\lambda}{|\log \lambda|},\quad u_\lambda(p_i)= -4\log \lambda-2\log \left(\frac{1}{\sqrt{2}}\log \frac{1}{\lambda} \right)+O(1).}
In particular if $f$ has exactly $m$ non-degenerate global minimum points, then $2^m$ distinct large solutions exist for all sufficiently small $\la$.\newline

In order to state our main result, we consider the singular problem
\eq{\Delta_g G-fe^G+8\pi\sum_{i=1}^{n} \delta_{p_i}+\alpha=0,\quad \mbox{in }M, \label{Green}}
where $\delta_{p_i}$ designates the Dirac mass at the point $p_i$. We have the following result.
\begin{lema}\label{lema1}
Problem (\ref{Green}) has a unique solution $G$ which is smooth away from the singularities and in local conformal coordinates around $p_i$ it has the form
\eq{G(x)=-4\log |x|-2\log \left(\frac{1}{\sqrt{2}}\log \frac{1}{|x|} \right)+\mathcal{H}(x),\label{Gr}}
where $\mathcal{H}(x)\in C(M)$.
\end{lema}
Our main result is the following.
\begin{teo}\label{main}
Let $p_1,\ldots,p_n$ be points such that
$f(p_i)=0$ and $D^2f(p_i)$ is positive definite for each $i$. Then, there exists a family of solutions $u_\lambda$ to (\ref{EQ}) with
\equ{\lambda^2 e^{u_\lambda}\rightharpoonup8\pi \sum_{i=1}^n \delta_{p_i},\quad \mbox{as }\la \rightarrow 0,}
and $u_\la \to G$ uniformly in compacts subsets of $M\setminus \left\{p_1,\ldots,p_k \right\}$. We define
\equ{c_i=\frac12e^{\mathcal{H}(p_i)/2},\quad\delta_\lambda^i=\frac{c_i}{|\log\lambda|},\quad \varepsilon_\lambda^i=\lambda\delta_\lambda^i}
where $\mathcal{H}$ is defined near $p_i$ by relation (\ref{Gr}).
In local conformal coordinates around $p_i$, there holds
\equ{u_\lambda(\varepsilon_\la^i x)+4\log \lambda + 2\log \delta_\la^i\rightarrow \log\frac{8}{(1+|x|^2)^2},}
uniformly on compact sets of $\R^2$ as $\la\rightarrow 0$.
\end{teo}

Our proof consists of the construction of a suitable first approximation of a solution as
required, and then solving by linearization and a suitable Lyapunov-type reduction
There is a large literature in Liouville type equation in two-dimensional domains or compact manifold, in particular concerning construction and classification of
blowing-up families of solutions. See for instance \cite{L0,L1,L2,L3,L4,L5} and their references.

\medskip
We shall present the detailed proof of our main result in the case of one bubbling point $n=1$.  In the last section we explain the necessary (minor, essentially notational) changes  for general $n$.  Thus, we consider the problem
\eq{\Delta_g u-fe^u+\lambda^2 e^u +\alpha=0,\quad \mbox{in }M, \label{problem}}
under the following hypothesis: there exists a point $p\in M$ such that $f(p)=0$ and $D^2f(p)$ is positive definite.

\section{A nonlinear Green's function}{\label{S3}}
We consider the singular problem
\eq{\Delta_g G-fe^G+8\pi \delta_p+\alpha=0,\quad \mbox{in }M \label{ProbLim}}
where $\delta_p$ is the Dirac mass supported at $p$, which is assume to be a point of global non-degenerate minimum of $f$. In this section we will establish the following result,
which corresponds to the case $n=1$ in Lemma \ref{lema1}

\begin{lema}\label{lema2}
Problem (\ref{ProbLim}) has a unique solution $G$ which is smooth away from the singularities and in local conformal coordinates around $p$ it has the form
\eq{G(x)=-4\log |x|-2\log \left(\frac{1}{\sqrt{2}}\log \frac{1}{|x|} \right)+\mathcal{H}(x),\label{Gr}}
where $\mathcal{H}(x)\in C(M)$.
\end{lema}

\begin{proof}

In order to construct a solution to this problem, we first consider the equation, in local conformal coordinates around $p$, for $\gamma \ll 1$
\eq{
\Delta \Gamma-fe^\Gamma+8\pi \delta_0=0,\quad \mbox{in } B(0,\gamma).\label{eqGamma}}
Since
\equ{
-\Delta \log \frac{1}{|x|^4}=8\pi \delta_0,
}
we look for a solution of (\ref{eqGamma}) of the form $\Gamma= -4\log |x|+h(x)$, where $h$ satisfies
\eq{\label{GreenLocal}
\Delta h-f(x)\frac{1}{|x|^4}e^h=0,\quad \mbox{in } B(0,\gamma).}
Since $p$ is a non-degenerate point of minimum of $f$, we may assume that, in local conformal coordinates around $p$, there exist positive numbers $\beta_1,\beta_2,\gamma$ such that
\eq{\beta_1|x|^2\leq f(x)\leq \beta_2 |x|^2,\label{nonDeg}}
for all $x\in B(0,\gamma)$.
Letting $r=|x|$, it is thus important to consider the equation
\eq{\Delta V-\frac{1}{r^2}e^V=0,\quad \mbox{in } B(0,\gamma). \label{eqV}}
For a radial function $V=V(r)$, this equation becomes
\eq{V''(r)+\frac{1}{r}V'(r)-\frac{1}{r^2}e^{V(r)}=0,\quad 0<r<\gamma. \label{eqV2}}
We make the change of variables $r=e^t$, $v(t)=V(r)$, so that equation (\ref{eqV2}) transforms into
\equ{\frac{d^2}{d t^2}v(t)=e^{v(t)},\quad -\infty <t<\log \gamma.}
from where it follows that
\equ{\frac{d}{dt}\left(\frac{v'(t)^2}{2}-e^{v(t)}\right)=0,}
or $v'(t)^2=2(e^v+C)$, for some constant $C$. Choosing $C=0$, we have
\equ{\frac{d}{dt}\left(e^{-v(t)/2}\right)=-\frac{1}{\sqrt{2}}.}
Integrating and coming back to the original variable, we deduce that
\equ{V(r)=-2\log\left(\frac{1}{\sqrt{2}} \log \frac{1}{r} \right)}
is a radial solution of equation (\ref{eqV}).
From condition (\ref{nonDeg}) we readily find that $h_1(x)=V(|x|)-\log\beta_1$ is a supersolution of (\ref{GreenLocal}), while $h_2(x)=V(|x|)-\log\beta_2$ is a subsolution of (\ref{GreenLocal}). This suggest us to look a solution to (\ref{GreenLocal}) of the form $V(|x|)+O(1)$.

Now we deal with existence of a solution of problem (\ref{ProbLim}). The previous analysis suggest that the singular part of the Green's function, in local conformal coordinates around $p$, is
\equ{\Gamma(x):=-4\log |x|+V(|x|),}
so we look for a solution of (\ref{ProbLim}) of the form $u=\eta \Gamma + H$, where $\eta$ is a smooth cut-off function such that $\eta\equiv 1$ in $B(0,\frac\gamma2)$ and $\eta\equiv 0$ in $\R^2\setminus B(0,\gamma)$. Therefore, $H$ satisfies the equation
\eq{\Delta_g H-fe^{\eta \Gamma}e^H+\alpha=-\eta fe^\Gamma -2\nabla_g \eta \nabla_g \Gamma -\Gamma \Delta_g \eta=:\Theta,\quad \mbox{in }M\label{eqReg}.}
Observe that $fe^{\eta \Gamma}\in L^1(B(0,\gamma))$. Next we find ordered global sub and supersolutions for (\ref{eqReg}). Let us consider the problem
\equ{-\Delta_g h_0+fh_0=1,\quad\mbox{in }M,}
which has a unique non-negative solution of class $C^{2,\beta}, 0<\beta<1$. Observe that
\equ{\Delta_g \beta h_0-fe^{\eta \Gamma}e^{\beta h_0}+\alpha-\Theta=-\beta + f\beta h_0 -fe^{\eta \Gamma}e^{\beta h_0}+\alpha-\Theta,}
so if we choose $\beta=\beta_1<0$ small enough, then $\underline{H}:=\beta_1h_0$ is a subsolution of (\ref{eqReg}), while if we choose $\beta=\beta_2>0$ large enough, then $\overline{H}:=\beta_2h_0$ is a supersolution of (\ref{eqReg}).

We consider the space
\equ{X=\left\{H\in H^1(M,g) \ \vline \ \int_Mfe^{\eta \Gamma}e^H <  \infty \right\},}
and the energy functional
\eq{E(H)=\frac12\int_M|\nabla_g H|^2 + \int_M fe^{\eta \Gamma}F(H)+\int_M(-\alpha+\Theta) H,\label{energ}}
where
\equ{
F(H(x))=\left\lbrace
\begin{array}{cl}
e^{\underline{H}(x)}(H-\underline{H}(x))&\quad H<\underline{H}(x),\\
e^H-e^{\underline{H}(x)}&\quad H\in[\underline{H}(x),\overline{H}(x)],\\
e^{\overline{H}(x)}(H-\underline{H}(x))&\quad H>\overline{H}(x).
\end{array}\right.}
Observe that since $h_0\in L^{\infty}(M,g)$ and $fe^{\eta \Gamma}\in L^1(B(p,\gamma))$, then $\overline{H},\underline{H}\in X$, which means that the functional $E$ is well defined in $X$.
Since
\equ{\int_M -\Delta_g(\eta \Gamma)=-\lim_{a\rightarrow 0}\int_{\partial B(p,a)} \frac{\partial \Gamma}{\partial r}=8\pi,}
we conclude that
\equ{\int_M\Theta =\int_M(-\Delta_g(\eta \Gamma)-8\pi \delta_p)=0.}

Besides $\alpha>0$, so the functional $E$ is coercive in $X$. We claim that $E$ attains a minimum in $X$. In fact, taking $H_n\in X$ such that
\equ{\lim_{n\rightarrow \infty} E(H_n)=\inf_{H\in X}E(H)>-\infty,}
and passing to a subsequence if necessary, we obtain
\equ{H_n\rightarrow \mathcal{H}\in X \mbox{ (in $L^2$)},\, \nabla_g H_n \rightharpoonup \nabla_g \mathcal{H} \mbox{ (weakly in $L^2$)},\, E(\mathcal{H})=\inf_{H\in X}E(H).}
Observe that if we take $\varphi\in C^\infty(M)$ then $\mathcal{H}+\varphi \in X$,  we can differentiate and obtain
\equ{\left. \frac{\partial}{\partial t}E(\mathcal{H}+t\varphi)\right|_{t=0}=0,\quad \mbox{for all } \varphi \in C^\infty (M,g) } or
\eq{\int_M \nabla_g \mathcal{H}\cdot \nabla_g \varphi +
\int_M fe^{\eta \Gamma} G(\mathcal{H})\varphi+
\int_M (-\alpha+\Theta)\varphi=0,\label{Eqenergy}}
where
\equ{
G(H)=\left\lbrace
\begin{array}{cl}
e^{\underline{H}(x)}&\quad H<\underline{H}(x),\\
e^H&\quad H\in[\underline{H}(x),\overline{H}(x)],\\
e^{\overline{H}(x)}&\quad H>\overline{H}(x).
\end{array}\right.}
By suitably approximating
$H_1=(\underline{H}-\mathcal{H})_+ $, we can use it as a test function in
\eqref{Eqenergy} and obtain
\equ{\int_M \nabla_g \mathcal{H}\cdot \nabla_g H_1 +
\int_M fe^{\eta \Gamma} G(\mathcal{H})H_1+
\int_M (-\alpha+\Theta)H_1=0.}
Since $\underline{H}$ is a subsolution for Equation (\ref{eqReg}),  we have
\equ{\int_M \nabla_g \underline{H}\cdot \nabla_g H_1 +
\int_M fe^{\eta \Gamma} e^{\underline{H}}H_1+
\int_M (-\alpha+\Theta)H_1\leq 0.}
Observe that
\equ{\int_M fe^{\eta \Gamma} G(\mathcal{H})H_1=\int_M fe^{\eta \Gamma}e^{\underline{H}}H_1.}
From the above calculations we deduce
\equ{\int_M |\nabla_g H_1|^2
\leq 0,}
hence $H_1\equiv C$ for some constant $C$. If $C>0$, necessarily $C\equiv H_1\equiv \underline{H}-\mathcal{H}$ almost everywhere. Thus, $\underline{H}=\mathcal{H}+C$, and  (\ref{Eqenergy}) traduces into
\equ{\int_M \nabla_g \underline{H}\cdot \nabla_g \varphi +
\int_M fe^{\eta \Gamma} e^{\underline{H}}\varphi+
\int_M (-\alpha+\Theta)\varphi=0,}
for all $\varphi\in C^\infty(M)$, which contradicts the fact that $\underline{H}$ solves
\equ{-\Delta_g \underline{H}+f\underline{H}=1,}
or in other words, the fact that $\underline{H}$ is not a solution of  problem (\ref{eqReg}). Hence $H_1\equiv 0$, which implies $\underline{H}\leq\mathcal{H}$.
In a similar way, we find $\mathcal{H} \le \overline{H}$ and hence
\equ{\underline{H}(x)\leq \mathcal{H}(x)\leq \overline{H}(x),\quad \hbox{a.e. } x\in M.}
Note that
\eq{\int_M \nabla_g \mathcal{H}\cdot \nabla_g \varphi+
\int_M fe^{\eta \Gamma} e^{\mathcal{H}}\varphi+
\int_M (-\alpha+\Theta)\varphi=0,\label{VariacionalH}}
for all $\varphi\in C^\infty(M,g)$.
Besides, since the functional $E$ is strictly convex and coercive, we conclude that $\mathcal{H}$ is the unique minimizer in $X$.

\medskip
So far we have proven that
Problem (\ref{ProbLim}) has a unique solution $G$ which is smooth away from the singularity point $p$ and in local conformal coordinates around $p$ it has the form
\equ{G(x)=\eta\left[-4\log |x|-2\log \left(\frac{1}{\sqrt{2}}\log \frac{1}{|x|} \right)\right]+\mathcal{H}(x),}
where $\mathcal{H}\in X\cap L^{\infty}(M,g)$, is the unique minimizer of the functional $E$ defined in $X$ by (\ref{energ}).

\medskip
  Next we will further study the form of $\mathcal{H}$ near $p$,  which in particular yields its continuity at $p$. For this purpose we use local conformal coordinates around $p$.

\medskip
Let us consider the problem
\equ{\left\lbrace
\begin{array}{rcll}
-\Delta_g \mathcal{J} &=&\alpha&\mbox{in }B(0,\frac\gamma2),\\
\mathcal{J}&=&\mathcal{H}&\mbox{on }\partial B(0,\frac\gamma2).
\end{array}\right.}
This problem has a unique solution $\mathcal{J}$, which is smooth in $B(0,\frac\gamma2)$. So we can expand $\mathcal{J}$ as
\equ{\mathcal{J}=\sum_{k=0}^\infty b_k r^k=b_0+O(r).}
We write $\mathcal{H}=\mathcal{J}+\mathcal{F}$, therefore $\mathcal{F}$ solves
\equ{\left\lbrace
\begin{array}{rcll}
\displaystyle-\Delta_g \mathcal{F}+\frac{f}{r^4}\frac{2}{\log^2 r}e^{\mathcal{J}}e^\mathcal{F}-\frac{1}{r^2}\frac{2}{\log^2 r} &=&0&\mbox{in }B(0,\frac\gamma2),\\
\mathcal{F}&=&0&\mbox{on }\partial B(0,\frac\gamma2),
\end{array}\right.}
because $\eta \Gamma\equiv\Gamma$ in $B(0,\frac\gamma2)$. Since $\mathcal{F}\in L^2(B(0,\frac\gamma2))$ we can expand it as
\equ{\mathcal{F}(r,\theta)=\sum_{k=0}^{\infty}a_k(r)e^{ik\theta}.}
Observe that
\equ{\frac{f(x)}{r^2}=\frac{\kappa_1r^2\cos^2(\theta)+\kappa_2r^2\sin^2(\theta)+\kappa_3r^2\sin\theta\cos\theta}{r^2}+O(r)=a(\theta)+O(r),}
for $r\neq 0$. Besides, $\beta_1\leq a(\theta)\leq \beta_2$.
Thus
\equ{\frac{f(x)}{r^4}\frac{2}{\log^2 r}e^{\mathcal{J}}e^\mathcal{F}-\frac{1}{r^2}\frac{2}{\log^2 r}=\frac{1}{r^2}\frac{2}{\log^2 r}\left[(a(\theta)+O(r))e^{\mathcal{J}+\mathcal{F}}-1\right].}
Moreover, since $\mathcal{H}\in L^\infty(B(0,\frac{\gamma}{2}))$ we have $e^{\mathcal{J}+\mathcal{F}}\in L^2(B(0,\frac\gamma2))$, so
\equ{\frac{1}{r^2}\frac{2}{\log^2 r}\left[(a(\theta)+O(r))e^{\mathcal{J}+\mathcal{F}}-1\right]=\sum_{k=0}^{\infty}m_k(r)e^{ik\theta},}
where
\equ{\left|m_k(r) \right|\leq \frac{C}{r^2}\frac{1}{\log^2 r},\quad \forall k\geq 0,}
for a constant $C$ independent of $k$. Now, we study the behavior of the coefficients $a_k(r)$. For this purpose let us remember that
\equ{\Delta u(r,\theta)=\frac{\partial^2 u}{\partial r^2}+\frac{1}{r}\frac{\partial u}{\partial r}+\frac{1}{r^2}\frac{\partial^2 u}{\partial \theta^2}.}
For $k\geq 1$, we see that $a_k(r)$ satisfies the ordinary differential equation
\eq{-\frac{\partial^2 a_k}{\partial r^2}(r)-\frac{1}{r}\frac{\partial a_k}{\partial r}(r)+\frac{k^2}{r^2} a_k(r)=m_k(r),\quad 0<r<\frac{\gamma}{2},\label{eqAk}}
under the conditions
\eq{a_k\left(\frac{\gamma}{2}\right)=0,\quad a_k(r)\in L^\infty\left(\left[0,\frac{\gamma}{2}\right]\right).\label{condAk}}
We recall that the $L^\infty$-condition comes from the fact that $\mathcal{F}\in L^\infty(B(0,\frac\gamma2))$.
Let us make the change of variables $r=e^t$, $A_k(t)=a_k(e^t)$, $M_k(t)=m_k(e^t)$, so the previous problem transform into
\eq{-\frac{d^2A_k}{dt^2}(t)+k^2A_k(t)=M_k(t),\quad -\infty<t<\log \frac\gamma2,\label{a1}}
under the conditions
\eq{A_k\left(\log \frac\gamma2\right)=0,\quad A_k\in L^{\infty}\left(\left(-\infty,\log \frac\gamma2 \right]  \right).\label{a2}}
Besides, $|M_k(t)|\leq C t^{-2}$ for all $k\geq 1$. All the solutions of the homogeneous equation are given by linear combinations of $e^{kt}$ and $e^{-kt}$ and a particular solution $A_k^{part}$ of the non-homogeneous equation (\ref{a1}) is given by the variation of parameter formula. We conclude that this problem has a solution of the form
\equ{C_1e^{kt}+C_2e^{-kt}+A_k^{part}.}
By the $L^\infty$-condition we conclude that $C_2=0$ and by the boundary condition in (\ref{a2}) we deduce $C_1=0$. This implies that the null function is the only solution of the homogeneous equation under condition (\ref{a2}). Hence, this problem has a unique solution $A_k(t)$. We claim that for a constant $C$ independent of $k$ we have
\eq{|A_k(t)|\leq C \frac{1}{k^2t^2}.\label{claim1}}
The proof of this fact is based on maximum principle: Observe that since $k^2>0$, the operator
\equ{-\frac{d^2}{dt^2}+k^2}
satisfies the weak maximum principle on bounded subsets of $(-\infty,\log\frac\gamma2]$.  Let us prove that $\phi=\frac{C_1}{k^2t^2}+\rho e^{-kt}$ is a non-negative supersolution for this problem. Observe first that since $A_k(t)$ is bounded, there exist $\tau_\rho$ such that
\equ{A_k(t)\leq\phi(t),\quad \foral t\in (-\infty,\tau_\rho].}
Besides,
\equ{\left(-\frac{d^2}{dt^2}+k^2 \right)\phi=-6C_1\frac{1}{k^2t^4}+C_1\frac{1}{t^2}\geq M_k(t),\quad \forall t\in \left(\tau_\rho,\log\frac\gamma2\right),}
where the last inequality is valid if we choose $C_1$ large enough. Observe also that $\phi(t)\geq A_k(t)$ for $t=\tau_\rho,\log\frac\gamma2$. Hence, by weak maximum principle we conclude that for all $\rho>0$
\equ{A_k(t)\leq \frac{C_1}{k^2t^2}+\rho e^{-kt},\quad \forall t\in \left(-\infty,\log \frac\gamma2\right].}
Taking the limit $\rho\rightarrow 0$ in the last expression, we conclude that $A_k(t)\leq C \frac{1}{k^2t^2}$.
Analogously, we now prove that $\varphi=-\frac{C_2}{k^2t^2}-\rho e^{-kt}$ is a non-positive subsolution for this problem. Since $A_k(t)$ is bounded, there exist $\tau_\rho$ such that
\equ{\varphi(t)\leq A_k(t),\quad \forall t\in (-\infty,\tau_\rho].}
Besides,
\equ{\left(-\frac{d^2}{dt^2}+k^2 \right)\varphi=6C_2\frac{1}{k^2t^4}-C_2\frac{1}{k^2t^2}\leq M_k(t),\quad \forall t\in \left(\tau_\rho,\log\frac\gamma2\right),}
where the last inequality is valid if we choose $C_2$ large enough. Observe also that $\varphi(t)\leq A_k(t)$ for $t=\tau_\rho,\log\frac\gamma2$. Hence, by weak maximum principle we conclude that for all $\rho>0$
\equ{-\frac{C_2}{k^2t^2}-\rho e^{-kt}\leq A_k(t),\quad \forall t\in \left(-\infty,\log \frac\gamma2\right].}
Taking the limit $\rho\rightarrow 0$ in the last expression, we conclude (\ref{claim1}).
Finally, coming back to the variable $r$ we conclude that there exist a unique solution $a_k(r)$ of problem (\ref{eqAk})-(\ref{condAk}), and for a constant $C$ independent of $k$ we have
\equ{|a_k(r)|\leq C\frac{1}{k^2\log^2r},\quad 0<r<\frac\gamma2.}
Now we deal with $a_0(r)$. Observe that
\equ{e^\mathcal{F}=e^{a_0(r)}\left(1+O\left(\frac{1}{\log^2 r}  \right)  \right),\quad e^{\mathcal{J}}=e^{b_0}(1+O(r)),}
and
\equ{a(\theta)=\alpha_0+\sum_{k=1}^{\infty}\alpha_ke^{ik\theta},\quad \mbox{with }\alpha_0>0,}
so we conclude that $a_0(r)$ satisfies the ordinary differential equation
\equ{-\frac{\partial^2 a_0(r)}{\partial r^2}-\frac{1}{r}\frac{\partial a_0(r)}{\partial r}+2\frac{\alpha_0e^{b_0}e^{a_0(r)}-1}{r^2\log^2 r}=O\left(\frac{1}{r^2\log^4r}\right),}
under the following conditions
\equ{a_0\left(\frac{\gamma}{2}\right)=0,\quad a_0\in L^\infty\left(\left[0,\frac\gamma2\right]\right).}
We make the change of variables $r=e^t$, $\tilde{a}_0(t)=a_0(e^t)$, so the previous problem transform into
\eq{-\frac{d^2 \tilde{a_0}}{dt^2}+2\frac{\alpha_0e^{b_0}e^{\tilde{a}_0}-1}{t^2}=O\left(\frac{1}{t^4}\right),\label{eqa_0}}
under the conditions
\eq{\tilde{a}_0\left(\log \frac{\gamma}{2}\right)=0,\quad \tilde{a}_0\in L^\infty\left( \left(-\infty,\log \frac{\gamma}{2}\right]\right).\label{eqa_1}}
The $L^\infty$-condition implies that there exist a sequence $t_n\rightarrow -\infty$ such that
\equ{\tilde{a}_0(t_n)\rightarrow L,\quad \mbox{as }n\rightarrow\infty,}
where $L=-\log (\alpha_0 e^{b_0})$. If not there exist $M<0$ such that
\equ{|\alpha_0e^{b_0}e^{\tilde{a}_0}-1|\geq\epsilon >0,\quad \forall t<M,}
which means that
\equ{\left|\frac{d^2 \tilde{a_0}}{dt^2}\right|\geq C\frac{\epsilon}{t^2}, \quad \forall t<M.}
Thus
\equ{|\tilde{a}_0|\geq C\epsilon \log |t|, \quad \forall t<M,}
so $\tilde{a}_0$ is unbounded, a contradiction.

We claim that the problem (\ref{eqa_0}), (\ref{eqa_1}) has at most one solution. In fact, let us suppose by contradiction that $u_1$ and $u_2$ are two diferent solutions. We define $u=u_1-u_2$, which satisfies the problem
\equ{-\frac{d^2u}{dt^2}+2\alpha_0e^{b_0}c(t)u=0,}
under the conditions,
\equ{u\left(\log \frac{\gamma}{2}\right)=0,\quad u\in L^\infty\left( \left(-\infty,\log \frac{\gamma}{2}\right]\right),}
and where
\equ{c(t)=\left\lbrace
\begin{array}{cr}
0&\mbox{if }u_1(t)=u_2(t),\\
\frac{1}{t^2}\frac{e^{u_1(t)-u_2(t)}}{u_1(t)-u_2(t)}&\mbox{if }u_1(t)\neq u_2(t).
\end{array}
\right.}
Observe that $c(t)\geq 0$, so we can apply the strong maximum principle in bounded domains for this problem. Moreover, from the $L^\infty$ condition we deduce that there exists a sequence $t_n$ such that $u(t_n)\rightarrow 0$ as $n\rightarrow \infty$ (the proof of this fact is the same that we gave before). From this two facts, we deduce easily that $u_1\equiv u_2$.

Let us make the change of variables $-t=e^s$, $A_0(s)=\tilde{a}_0(-e^s)$, so the previous problem transform into
\eq{\label{prA0}-\frac{d^2 A_0}{ds^2}+\frac{dA_0}{ds}+2(\alpha_0e^{b_0}e^{A_0}-1)=O(e^{-2s}),}
under the conditions
\equ{A_0\left(\log \left(-\log\frac{\gamma}{2}\right)\right)=0,\quad A_0\in L^\infty\left(\left[\log \left(-\log\frac{\gamma}{2}\right),\infty\right)\right).}
We look for a solution of this problem of the form $A_0(s)=L+\phi(s)$, so $\phi$ solves the differential equation
\equ{-\frac{d^2\phi}{ds^2}+\frac{d\phi}{ds}+2\phi=N(\phi)+O(e^{-2s}),}
where
\equ{N(\phi)=-2(e^{\phi}-\phi-1).}
Observe that $\phi_+=e^{2s}$, $\phi_-=e^{-s}$ are two linear independent solutions of the homogeneous equation.

From the previous analysis, we deduce that there exists a sequence $s_n\rightarrow \infty$ such that $\phi(s_n)=\delta_n \rightarrow 0$, as $n\rightarrow \infty$. We make the change of variables $\tilde{\phi}_n(\tau_n)=\phi(s)-\delta_n \phi_-(\tau_n)$, where $\tau_n=s-s_n$, so $\tilde{\phi}_n\in L^{\infty}$ solves the problem
\eq{\label{fixed}\left\lbrace
\begin{array}{rcll}
-\tilde{\phi}_n''+\tilde{\phi}_n'+2\tilde{\phi}_n&=&N(\tilde{\phi}_n+\delta e^{-\tau_n})+e^{-2s_n}O(e^{-2\tau_n})&\tau_n\in(0,\infty),\\
\tilde{\phi}_n(0)&=&0.&
\end{array}\right.}

Let us study the linear problem
\equ{\left\lbrace
\begin{array}{rcll}
-\varphi''+\varphi'+2\varphi&=&\omega&\mbox{in }(0,\infty),\\
\varphi(0)&=&0,&\varphi\in L^\infty(0,\infty)
\end{array}\right.}
for $\omega\in C([0,\infty))$ given. This problem has an explicit and unique solution $\varphi=T[g]$, in fact
\equ{\varphi(t)=C_1e^{\lambda_+t}+C_2e^{\lambda_-t}+e^{\lambda_+t}\int_0^t\frac{e^{\lambda_-s}\omega(s)}{3e^{2s}}ds-e^{\lambda_-t}\int_0^t\frac{e^{\lambda_+s}\omega(s)}{3e^{2s}}ds}
and we deduce that $C_1=0$ and $C_2=0$ due to the $L^\infty$ condition and the value at $0$ of $\varphi$, respectively. Problem (\ref{fixed}) can be written as \eq{\label{fix2}\tilde{\phi}_n=T[N(\tilde{\phi}_n+\delta e^{-\tau_n})+e^{-2s_n}O(e^{-2\tau_n})]:=A[\tilde{\phi}_n].}
We consider the set \equ{B_\epsilon=\left\{\varphi \in C([0,\infty))\,:\,\|\varphi\|_\infty\leq \epsilon  \right\}.} It is easy to see that if $s_n$ is large enough and $\delta_n$ small enough we have
\equ{\|A[\tilde{\phi}^1_n]-A[\tilde{\phi}^2_n]\|_\infty \leq C\epsilon \|\tilde{\phi}^1_n-\tilde{\phi}^2_n\|,}
\equ{\|A[\tilde{\phi}_n]\|\leq C\epsilon,}
and where $C$ is independent of $n$. It follows that for all sufficiently small $\epsilon$ we get that $A$ is a contraction mapping of $B_{\epsilon}$ (provided $n$ large enough), and therefore a unique fixed point of $A$ exists in this region. We deduce that there exists a unique solution $A_0$ of problem (\ref{prA0}), and it has the form $A_0(s)=L+\phi(s)$, where $L$ is a fixed constant, and $\phi(s)\rightarrow 0$ as $s\rightarrow \infty$. This concludes the proof of Lemma \ref{lema2}.
\end{proof}
\section{ Construction of a first approximation}\label{SecAprox}
In this section we will build a suitable approximation for a solution of Problem (\ref{problem}) which is large exactly near the point $p$.
The ``basic cells'' for the construction of an approximate solution of problem (\ref{problem}) are the radially symmetric solutions of the problem
\eq{\label{BasicCell}\left\lbrace
\begin{array}{rcll}
\Delta w+\lambda^2 e^w&=&0&\mbox{in }\R^2,\\
w(x)&\rightarrow&0&\mbox{as }|x|\rightarrow \infty.
\end{array}\right.}
which are given by the one-parameter family of functions
\equ{w_\delta(|x|)=\log \frac{8\delta^2}{(\lambda^2\delta^2+|x|^2)^2},}
where $\delta$ is any positive number. We define $\varepsilon=\lambda \delta$. In order to construct the approximate solution we consider the equation
\eq{\Delta F-\frac{\delta^2}{r^2}e^F=0,\label{eqAprox}}
in the variable $r=|x|/\varepsilon$ and we look for a radial solution $F=F(r)$, away from $r=0$. For this purpose we solve Problem (\ref{eqAprox}) under the following initial conditions
\equ{F(1/\delta)=0,\quad F'(1/\delta)=0.}
We make the change of variables $r=e^t$, $V(t)=F(r)$, so that equation \eqref{eqAprox} transforms into
\equ{V''-\delta^2e^V=0.}
We consider the transformation $V(s)=\tilde{V}(\delta s)$, so $\tilde{V}$ solves problem
\equ{\tilde{V}''-e^{\tilde{V}}=0,\quad \tilde{V}(\delta |\log \delta|)=0,\quad \tilde{V}'(\delta |\log \delta|)=0.}
This problem has a unique regular solution, which blows-up at some finite radius $\gamma>0$. Coming back to the variable $r=|x|/\varepsilon$, we conclude that the solution $F(r)$ is defined for all $1/\delta\le r \le Ce^{1/\delta}=C/\lambda$, for some constant $C$. Besides, we extend by $0$ the function $F$ for $r\in[0,1/\delta)$, which means $F(r)=0,$ for all $r\in[0,1/\delta)$ and we denote by $\tilde{F}(|x|)=F(|x|/\varepsilon)$. A first local approximation of the solution, in local conformal coordinates around $p$, is given by the radial function $u_\varepsilon(x)=w_\delta(|x|)+\tilde F(|x|)$.

In order to build a global approximation, let us consider $\eta$ a smooth radial cutoff function such that $\eta(r)=1$ if $r\leq C_1\delta$ and $\eta(r)=0$ if $r\geq C_2\delta$, for constants $0<C_1<C_2$. We consider as initial approximation $U_\varepsilon=\eta u_\varepsilon +(1-\eta)G$, where $G$ is the Green function that we built in the previous section. In order to have a good approximation around $p$ we have to adjust the parameter $\delta$. The good choice of this number is
\equ{\log 8\delta^2=-2\log\left(\frac{1}{\sqrt 2}\log \frac{1}{\lambda}\right)+\mathcal{H}(p),}
where $\mathcal{H}$ is defined in Section \ref{S3}.
With this choice of the parameter $\delta$, the function $u_\varepsilon$ is approaching the Green function $G$ around $p$.

A useful observation is that $u$ satisfies problem (\ref{problem})
if and only if
\equ{v(y)=u(\varepsilon y)+4\log \lambda +2\log\delta}
satisfies
\eq{\Delta_g v-\lambda^{-2}f(\varepsilon y)e^{v}+e^{v}+\varepsilon^2\alpha=0,\quad y\in M_\varepsilon,\label{probEscalado}}
where $M_\varepsilon=\varepsilon^{-1}M$.

We denote in what follows $p'=\varepsilon^{-1}p$ and
\equ{\tilde{U}_\varepsilon(y)=U_\varepsilon(\varepsilon y)+4\log \lambda +2\log \delta,}
for $y\in M_\varepsilon$. This means precisely in local conformal coordinates around $p$ that
\begin{align*}
\tilde{U}_\varepsilon(y)=&\eta(\varepsilon|y|)\left(\log \frac{1}{(1+|y|^2)^2}+\tilde{F}(\varepsilon |y|)\right)\\
&+(1-\eta(\varepsilon|y|))\left(G(\varepsilon y)+4\log\lambda + 2\log\delta \right).
\end{align*}
Let us consider a vector $k\in \R^2$. We recall that $w_\delta(|x-k|)$ is also a solution of problem (\ref{BasicCell}). In order to solve problem (\ref{probEscalado}), we need to modify the first approximation of the solution, in order to have a new parameter related to translations. More precisely, we consider for $|k|\ll 1$ the new first approximation of the solution (in the expanded variable)
\begin{align*}
V_\varepsilon(y)=&\eta(\varepsilon|y|)\left(\log \frac{1}{(1+|y-k|^2)^2}+\tilde{F}(\varepsilon |y|)\right)\\
&+(1-\eta(\varepsilon|y|))\left(G(\varepsilon y)+4\log\lambda + 2\log\delta \right).
\end{align*}
We will denote by $v_\varepsilon$ the first approximation of the solution in the original variable, which means
\equ{v_\varepsilon(x)=\eta(|x|)\left(\log \frac{8\delta^2}{(\varepsilon^2+|x-\varepsilon k|^2)^2}+\tilde{F}(|x|)\right)+(1-\eta(|x|))G(x).}


Hereafter we look for a solution of problem (\ref{probEscalado}) of the form $v(y)=V_\varepsilon(y)+\phi(y)$, where $\phi$ represent a lower order correction. In terms of $\phi$, problem (\ref{probEscalado}) now reads
\eq{\label{ProbPhi} L(\phi)=N(\phi)+E,\quad\mbox{in }M_\varepsilon, }
where
\begin{align*}
L(\phi)\colonequals&\Delta_g \phi -\lambda^{-2}f(\varepsilon y)e^{V_\varepsilon}\phi+e^{V_\varepsilon}\phi,\\
N(\phi)\colonequals&\lambda^{-2}f(\varepsilon y)e^{V_\varepsilon}(e^\phi-1-\phi)-e^{V_\varepsilon}(e^\phi-1-\phi),\\
E\colonequals&-(\Delta_g V_\varepsilon-\lambda^{-2}f(\varepsilon y)e^{V_\varepsilon}+e^{V_\varepsilon}+\varepsilon^2\alpha).
\end{align*}


\section{The linearized operator around the first approximation}
In this section we will develop a solvability theory for the second-order linear operator $L$ defined in (\ref{ProbPhi}) under suitable orthogonality conditions. Using local conformal coordinates around $p'$, then formally the operator $L$ approaches, as $\varepsilon,|k|\rightarrow 0$, the operator in $\R^2$
\equ{\mathcal{L}(\phi)=\Delta \phi + \frac{8}{(1+|z|^2)^2}\phi,}
namely, equation $\Delta w+e^w=0$ linearized around the radial solution $w(z)=\log \frac{8}{(1+|z|^2)^2}$. An important fact to develop a satisfactory solvability theory for the operator $L$ is the non-degeneracy of $w$ modulo the natural invariance of the equation under dilations and translations. Thus we set
\begin{align}
z_0(z)&=\frac{\partial}{\partial s}[w(sz)+2\log s]|_{s=1},\\
z_i(z)&=\frac{\partial}{\partial \zeta_i}w(z+\zeta)|_{\zeta=0},\quad i=1,2.
\end{align}
It turns out that the only bounded solutions of $\mathcal{L}(\phi)=0$ in $\R^2$ are precisely the linear combinations of the $z_i$, $i=0,1,2$, see \cite{f} for a proof. We define for $i=0,1,2$,
\equ{Z_i(y)=z_i(y-k).}
Additionally, let us consider $R_0$ a large but fixed number $R_0>0$ and $\chi$ a radial and smooth cut-off function such that $\chi\equiv 1$ in $B(k,R_0)$ and $\chi\equiv 0$ in $B(k,R_0+1)^c$.

Given $h$ of class $C^{0,\beta}(M_\varepsilon)$, we consider the linear problem of finding a function $\phi$ such that for certain scalars $c_i,\,i=1,2$, one has
\eq{\left\lbrace
\begin{array}{rcll}
L(\phi)&=&h+\sum_{i=1}^{2}c_i\chi Z_i&\mbox{in }M_\varepsilon,\\
\displaystyle \int_{M_\varepsilon} \chi Z_i \phi&=&0&\mbox{for }i=1,2.
\end{array}\right.\label{ProbPhi2}}
We will establish a priori estimates for this problem. To this end we define, given a fixed number $0<\sigma<1$, the norm
\eq{\|h\|_{*} = \|h\|_{*,p}\colonequals \sup_{ M_\varepsilon} ( \max\{\varepsilon^2, |y|^{-2-\sigma} \} )^{-1} |h|. \label{norma} }
Here the expression  $\max\{\varepsilon^2, |y|^{-2-\sigma} \}$ is regarded in local conformal coordinates around $p'= \ve^{-1}p$. Since local coordinates are defined up to distance $\sim \frac 1\ve $ that expression  makes sense globally in $M_\ve$.

Our purpose in this section is to prove the following result.
\begin{prop}{\label{PropLin}} There exist positive numbers $\varepsilon_0,C$ such that for any $h\in C^{0,\beta}(M_\varepsilon)$, with $\| h\|_{*}< \infty$ and for all k such that $|k|\leq C\lambda/\delta$, there is a unique solution $\phi=T(h)\in C^{2,\beta}(M_\varepsilon)$ of problem (\ref{ProbPhi2}) for all $\varepsilon<\varepsilon_0$, which defines a linear operator of $h$. Besides,
\eq{\|T(h)\|_\infty\leq C\log\left(\frac{1}{\varepsilon}\right)\|h\|_{*}.\label{propp}}
\end{prop}
Observe that the orthogonality conditions in problem (\ref{ProbPhi2}) are only taken respect to the elements of the approximate kernel due to translations.

The next Lemma will be used for the proof of Proposition \ref{PropLin}.
We obtain an a priori estimate for the problem
\eq{\left\lbrace
\begin{array}{rcll}
L(\phi)&=&h&\mbox{in }M_\varepsilon,\\
\displaystyle \int_{M_\varepsilon} \chi Z_i \phi&=&0&\mbox{for }i=1,2.
\end{array}\right.\label{ProbPhi3}}
We have the following estimate.
\begin{lema}
There exist positive constants $\varepsilon_0,\,C$ such that for any $\phi$ solution of problem (\ref{ProbPhi3}) with $h\in C^{0,\beta}(M_\varepsilon)$, $\|h\|_*<\infty$ and any k, $|k|\leq C\lambda/\delta$
\equ{\|\phi\|_\infty \leq C\log\left(\frac{1}{\varepsilon} \right) \|h\|_*,}
for all $\varepsilon<\varepsilon_0$.
\end{lema}

\begin{proof}
We carry out the proof by a contradiction argument. If the above fact were false, there would exist sequences $(\varepsilon_n)_{n\in\N}$, $(k_n)_{n\in \N}$ such that $\varepsilon_n\rightarrow 0$, $|k_n|\rightarrow 0$ and functions $\phi_n$, $h_n$ with $\|\phi_n\|_\infty=1$, \equ{\log(\varepsilon_n^{-1})\|h_n\|_{*}\rightarrow 0,}
such that
\eq{\left\lbrace
\begin{array}{rcll}
L(\phi_n)&=&h_n&\mbox{in }M_{\varepsilon_n},\\
\int_{M_{\varepsilon_n}}\chi Z_i \phi_n &=&0& \mbox{for }i=1,2.
\end{array}\right.\label{Prob_n}}
A key step in the proof is the fact that the operator $L$ satisfies a weak maximum principle in regions, in local conformal coordinates around $p$, of the form $A_\varepsilon=B(p',\varepsilon^{-1}\gamma/2)\setminus B(p',R)$, with $R$ a large but fixed number. Consider the function $z_0(r)=\frac{r^2-1}{r^2+1}$, radial solution in $\R^2$ of
\equ{\Delta z_0 + \frac{8}{(1+r^2)^2}z_0=0.}
We define a comparison function
\equ{Z(y)=z_0(a|y-p'|), \quad y\in M_\varepsilon.}
Let us observe that
\equ{-\Delta Z(y)=\frac{8a^2(a^2|y-p'|^2-1)}{(1+a^2|y-p'|^2)^3}.}
So, for $100a^{-2}<|y-p'|<\varepsilon^{-1}\gamma/2$, we have
\equ{-\Delta Z(y)\geq 2 \frac{a^2}{(1+a^2|y-p'|^2)^2}\geq \frac{a^{-2}}{|y-p'|^4}.}
On the other hand, in the same region,
\equ{e^{V_\varepsilon(y)} Z(y)\leq C\frac{1}{|y-p'|^4}.}
Hence if $a$ is taken small and fixed, and $R>0$ is chosen sufficiently large depending on this $a$, then
\equ{\Delta Z+e^{V_\varepsilon}Z<0,\quad \mbox{in }A_\varepsilon.}
Since $Z>0$ in $A_\varepsilon$, we have
\equ{L(Z)<0, \quad \mbox{in }A_\varepsilon}
We conclude that $L$ satisfies weak maximum principle in $A_\varepsilon$, namely if $L(\phi)\leq 0$ in $A_\varepsilon$ and $\phi\geq0$ on $\partial A_\varepsilon$, then $\phi\geq 0$ in $A_\varepsilon$.
\newline

We now give the proof of the Lemma in several steps.
\newline

STEP 1. We claim that
\equ{\sup_{y\in M_{\varepsilon_n}\setminus B(p/\varepsilon_n,\rho/\varepsilon_n)}|\phi_n(y)|=o(1),}
where $\rho$ is a fixed number. In fact, coming back to the original variable by the transformation
\equ{\hat{\phi}_n(x)=\phi_n\left(\frac{x}{\varepsilon_n}\right),\quad x\in M.}
We can see that $\hat{\phi}_n$ satisfies the equation
\eq{\Delta_g \hat{\phi}_n-fe^{v_{\varepsilon_n}}\hat{\phi}_n+\lambda_n^2e^{v_{\varepsilon_n}}\hat{\phi}_n=\frac{1}{\varepsilon_n^2}h_n\left(\frac{x}{\varepsilon_n}\right),\label{eqVarOriginal}}
where
\equ{v_{\varepsilon_n}(x)=V_{\varepsilon_n}\left(\frac{x}{\varepsilon_n}\right)-4\log\lambda_n-2\log\delta,}
is the approximation of the solution in the original variable. Taking $n\rightarrow \infty$, we can see that $\hat{\phi}_n$ converges uniformly over compacts of $M\setminus\{p\}$ to a function $\hat{\phi}\in H^1(M)\cap L^\infty(M)$ solution of the problem
\eq{\label{probLimit}  \Delta_g \hat{\phi}-fe^{J}\hat{\phi}=0,\quad\mbox{in }M\setminus \{p\}}
where $J$ is the limit of $v_{\varepsilon_n}$.
We claim that $\hat{\phi}\equiv 0$, in fact, we consider the unique solution $\Phi$ of the problem
\equ{\Delta_g \Phi-\min\{fe^{J},1\}\Phi=-\delta_{p},\quad\mbox{in }M.
}

 Using local conformal coordinates around $p$ we expand
 $$\Phi(x)=-\frac{1}{2\pi}\log(|x|)+H(x)
  $$
   for $H$ bounded. Since $\hat{\phi}\in L^\infty(M)$, we conclude that for all sufficiently small $\epsilon$ and $\tau$ we have
\equ{|\hat{\phi}(x)|\leq \epsilon \Phi(x),\quad x\in \partial B(0,\tau).}
Multiplying (\ref{probLimit}) by $\varphi=(\hat{\phi}-\epsilon \Phi)_+$, and integrating by parts over $M_\tau=M\setminus U_\tau$, where $U_\tau$ is the neighborhood around $p$ under the local conformal coordinates that we used, we have
\equ{\int_{M_\tau} |\nabla_g \varphi|^2+\int_{M_\tau} fe^{J}\varphi^2 +\epsilon\int_{M_\tau}e^{J}\varphi\Phi=0.}
Since $\Phi\geq 0$, we have
\equ{\int_{M_\tau} |\nabla_g \varphi|^2+\int_{M_\tau} fe^{J}\varphi^2 \leq 0.}
Hence $\varphi=(\hat{\phi}-\epsilon \Phi)_+=0$ in $M_\tau$, so $\hat{\phi}\leq \epsilon \Phi$ in $M_\tau$. Multiplying by $\varphi=(\hat{\phi}+\epsilon\Phi)_-$ and integrating by parts, we have $(\hat{\phi}+\epsilon\Phi)_-=0$, thus
\equ{|\hat{\phi}(x)|\leq \epsilon \Phi(x),\quad x\in M_\tau.}
Taking $\epsilon\rightarrow 0$ and $\tau\rightarrow 0$, we conclude that $\hat{\phi}\equiv 0$.
\newline


STEP 2. Let us consider the transformation
\equ{\tilde{\phi}_n(y)=\phi_n(y+p_n').}
Thus $\tilde{\phi}_n$ satisfies the equation
\equ{\Delta_g \tilde{\phi}_n-\lambda_n^{-2}f(\varepsilon_n y+p_n)e^{V_{\varepsilon_n}(y+p_n')}\tilde{\phi}_n+e^{V_{\varepsilon_n}(y'+p_n')}=h_n(y+p_n'),}
in $M_{\varepsilon_n}- \{p_n'\}$. Taking the limit $n\rightarrow \infty$ in the last equation (and also in problem (\ref{Prob_n})), we see that $\tilde{\phi}_n$ converges uniformly over compacts of $M_{\varepsilon_n}-\{p'_n\}$ to a bounded solution $\tilde{\phi}$ of the problem
\equ{\mathcal{L}(\tilde{\phi})=0\quad \mbox{in } \R^2,\quad \int_{\R^2}\chi Z_i\tilde{\phi}=0,\quad i=1,2.}
Hence $\tilde{\phi}(x)=C_0Z_0(x)$.

In what follows we assume without loss of generality that $C_0\geq 0$. If $C_0<0$, we work with $-\phi_n$ instead of $\phi_n$ and the following analysis is also valid.
\newline

STEP 3. In this step we will construct a non-negative supersolution in the region, in local conformal coordinates around $p_n'$, $B_n=B(k_n,\rho)\setminus B(k_n,\varepsilon_n^{-1}\gamma/2), \rho>0$, where the weak maximum principle is valid. We work first in the case $C_0>0$. Let us consider the problem
\eq{\left\lbrace
\begin{array}{rcll}
-\Delta \psi_n-e^{V_\varepsilon}\psi_n&=&1&\mbox{in }B_n,\\
\psi_n(y)&=&C_0&\mbox{on  }\partial B(k_n,\rho),\\
\psi_n(y)&=&o(1)&\mbox{on  }\partial B(k_n,\varepsilon_n^{-1}\gamma/2).
\end{array}\right.\label{supersolucion}}
We define $r=|y-k_n|$.
A direct computation shows that
\equ{\psi_n(y)=C_0Z_0(r)+C_\varepsilon Y(r)+W(r),}
where
\equ{Y(r)=Z_0\int_\rho^r \frac{1}{sZ_0^2(s)}ds,\quad W(r)=-Z_0(r)\int_{\rho}^{r}sY(s)ds+Y(r)\int_{\rho}^{r}sZ_0(s)ds,}
and \equ{C_\varepsilon=\frac{o(1)-C_0Z_0(\varepsilon_n^{-1}\gamma/2)-W(\varepsilon_n^{-1}\gamma/2)}{Y(\varepsilon_n^{-1}\gamma/2)}.}
We choose  $\rho>R$, where $R$ is the fixed minimal radio for which the weak maximum principle is valid in the region $B_n$. Observe that
\equ{L(\psi_n)=-1-\lambda^{-2}f(\varepsilon y)e^{V_\varepsilon}\psi_n\leq h_n=L(\phi_n).}
Moreover, from steps 1 and 2, we deduce that
\eq{\psi_n\geq \phi_n, \quad\mbox{on } \partial B_n,\label{Fact0}}
which means that $\psi_n$ is a supersolution for the problem
\equ{L(\phi_n)=h_n,\quad \mbox{in }B_n.}
Since $\rho>R$, we can apply the weak maximum principle and we deduce that $\Psi_n\geq \phi_n$ in $B_n$. Observe that
\eq{\left| \frac{d\psi_n(\rho)}{dr}\right|\geq \varepsilon_n^{-1}.\label{Fact1}}
In the other hand
\eq{\frac{dZ_0}{dr}=-C\frac{r}{(r^2-1)^2},\label{Fact2}}
where $C>0$ is a constant independent of $n$. Since $\phi_n$ converges over compacts of the expanded variable to the function $C_0Z_0$, we deduce from (\ref{Fact0}), (\ref{Fact1}) and (\ref{Fact2}) that the partial derivative of $\phi_n$ respect to $r$ is discontinuous at $|y-k_n|=\rho$, for large values of $n$, which is a contradiction.

In the case $C_0=0$, $\phi_n$ converges to $0$ over compacts of the expanded variable. Let us consider the problem
\equ{\left\lbrace
\begin{array}{rcll}
-\Delta \psi_n-e^{V_\varepsilon}\psi_n&=&1&\mbox{in }B_n,\\
\psi_n(y)&=&1/2&\mbox{on  }\partial B(k_n,\rho),\\
\psi_n(y)&=&o(1)&\mbox{on  }\partial B(k_n,\varepsilon_n^{-1}\gamma/2).
\end{array}\right.}
It is easy to see that $\psi_n\leq 1/2$ in $\overline{B}_n$. Using the previous maximum principle argument we deduce that $\phi_n\leq\psi_n\leq 1/2$ Applying the same argument for the problem that $-\phi_n$ satisfies, we conclude $-\phi_n\leq 1/2$. Thus,
\equ{\|\phi_n\|_\infty\leq 1/2,}
which is a contradiction with the fact $\|\phi_n\|_\infty=1$.
This finishes the proof of the a priori estimate.
\end{proof}
We are now ready to prove the main result of this section.
\begin{proof}[Proof of Proposition \ref{PropLin}]
We begin by establishing the validity of the a priori estimate (\ref{propp}). The previous lemma yields
\eq{\|\phi\|_{\infty}\leq C \log \left(\frac{1}{\varepsilon}\right)\left[ \|h\|_{*}+\sum_{i=1}^2|c_i| \right],\label{test0}}
hence it suffices to estimate the values of the constants $|c_i|,\,i=1,2$. We use local conformal coordinates around $p$, and we define again $r=|y|$ and we consider a smooth cut-off function $\eta(r)$ such that $\eta(r)=1$ for $r<\frac{1}{\sqrt{\varepsilon}}$, $\eta(r)=0$ for $r>\frac{2}{\sqrt{\varepsilon}}$, $|\eta'(r)|\leq C\sqrt{\varepsilon}$, $|\eta''(r)|\leq C\varepsilon$. We test the first equation of problem (\ref{ProbPhi2}) against $\eta Z_i$, $i=1,2$ to find
\eq{\langle L(\phi),\eta Z_i\rangle=\langle h,\eta Z_i \rangle+c_i\int_{M_\varepsilon}\chi |Z_i|^2.\label{test1}}
Observe that
\equ{\langle L(\phi),\eta Z_i\rangle=\langle \phi,L(\eta Z_i)\rangle,}
and
\equ{L(\eta Z_i)=Z_i\Delta \eta + 2\nabla \eta \cdot \nabla Z_i+\eta(\Delta Z_i+e^{V_\varepsilon}Z_i)-\eta\lambda^{-2}f(\varepsilon y)e^{V_\varepsilon}Z_i.}
We have
\equ{\eta(\Delta Z_i+e^{V_\varepsilon}Z_i)=\varepsilon O((1+r)^{-3}).}
Observe that
\equ{\lambda^{-2}f(\varepsilon y)e^{V_\varepsilon(y)}=\lambda^2\delta^2 f(x)e^{v_\varepsilon(x)},\quad \mbox{where } y=\frac{x}{\varepsilon}, x\in M,}
thus
\equ{\eta\lambda^{-2}f(\varepsilon y)e^{V_\varepsilon}Z_i=O(\varepsilon^2).}
Since $\Delta \eta=O(\varepsilon)$, $\nabla \eta= O(\sqrt \varepsilon)$, and besides $Z_i=O(r^{-1})$, $\nabla Z_i=O(r^{-2})$, we find
\equ{Z_i\Delta \eta + 2\nabla \eta \cdot \nabla Z_i=O(\varepsilon\sqrt{\varepsilon}).}
From the previous estimates we conclude that
\equ{|\langle \phi,L(\eta Z_i) \rangle|\leq C\sqrt\varepsilon \|\phi\|_\infty.}
Combining this estimate with (\ref{test0}) and (\ref{test1}) we obtain
\equ{|c_i|\leq C\left[\|h\|_*+\sqrt\varepsilon\log \frac{1}{\varepsilon} \right],}
which implies
\equ{|c_i|\leq C\|h\|_*\quad i=1,2.}
It follows from (\ref{test0}) that
\equ{\|\phi\|_\infty\leq C \log \left(\frac{1}{\varepsilon}\right)\|h\|_*,}
and the a priori estimate (\ref{propp}) has been thus proven. It only remains to prove the solvability assertion. For this purpose let us consider the space
\equ{H=\left\lbrace  \phi\in H^1(M_\varepsilon)\, : \,  \int_{M_\varepsilon}\chi Z_i\phi=0,\, i=1,2.  \right \rbrace}
endowed with the inner product,
\equ{\langle\phi,\psi\rangle=\int_{M_\varepsilon}\nabla_g \phi \nabla_g \psi+\int_{M_\varepsilon}\lambda^{-2}f(\varepsilon y)e^{V_\varepsilon}\phi\psi.}
Problem (\ref{ProbPhi2}) expressed in weak form is equivalent to that of finding $\phi\in H$ such that
\equ{\langle\phi,\psi\rangle=\int_{M_\varepsilon}\left[e^{V_\varepsilon}\phi+h+\sum_{i=1}^2 c_i \chi Z_i\right]\psi,\quad \mbox{for all }\psi\in H.}

With the aid of Riesz's representation theorem, this equation gets rewritten in $H$ in the operator form $\phi=K(\phi)+\tilde{h}$, for certain $\tilde{h}\in H$, where $K$ is  a compact operator in $H$. Fredholm's alternative guarantees unique solvability of this problem for any $h$ provided that the homogeneous equation $\phi=K(\phi)$ has only zero as solution in $H$. This last equation is equivalent to problem (\ref{ProbPhi2}) with $h\equiv 0$. Thus, existence of a unique solution follows from the a priori estimate (\ref{propp}). The proof is complete.
\end{proof}

\section{The nonlinear problem}
We recall that our goal is to solve problem (\ref{ProbPhi}). Rather than doing so directly, we shall solve fist the intermediate problem
\eq{\left\lbrace
\begin{array}{rcll}
L(\phi)&=&N(\phi)+E+\sum_{i=1}^{2}c_i\chi Z_i&\mbox{in }M_\varepsilon,\\
\displaystyle \int_{M_\varepsilon} \chi Z_i \phi&=&0&\mbox{for }i=1,2,
\end{array}\right.\label{ProbNonLin}}
using the theory developed in the previous section. We assume that the conditions in Proposition (\ref{PropLin}) hold. We have the following result
\begin{lema}\label{lemita}
Under the assumptions of Proposition (\ref{PropLin}) there exist positive number $C,\varepsilon_0$ such that problem (\ref{ProbNonLin}) has a unique solution $\phi$ which satisfies
\equ{\|\phi\|_\infty \leq C\varepsilon \log \frac{1}{\varepsilon},}
for all $\varepsilon<\varepsilon_0$.
\end{lema}
\begin{proof}
In terms of the operator $T$ defined in Proposition (\ref{PropLin}), problem (\ref{ProbNonLin}) becomes
\eq{\phi=T(N(\phi)+E)=:A(\phi).\label{nonlin}}
For a given number $\vartheta>0$, let us consider the space
\equ{H_\vartheta=\left\{\phi \in C(M_\varepsilon)\,:\,\|\phi\|_\infty \leq \vartheta \varepsilon \log \frac{1}{\varepsilon} \right\}.}
From Proposition (\ref{PropLin}), we get
\equ{\|A(\phi)\|_\infty\leq C\log \left(\frac{1}{\varepsilon}\right)(\|N(\phi)\|_*+\|E\|_*).}
Let us first measure how well $V_\varepsilon$ solves problem (\ref{probEscalado}).
Observe that
\eq{e^{V_\varepsilon(y)}=\lambda^4\delta^2 e^{v_\varepsilon(x)},\quad y=\frac{x}{\varepsilon}, x\in M, \label{estimate}}
so
\equ{\|e^{V_\varepsilon(y)}\|_*\leq C\varepsilon.}
As a consequence of the construction of the first approximation, the choice of the parameter $\delta$, the expansion of the Green function $G$ around $p$, and \eqref{estimate}, a direct computation yields
\equ{\|E\|_*\leq C\varepsilon.}
Now we estimate
\equ{N(\phi)=\lambda^{-2}f(\varepsilon y)e^{V_\varepsilon}(e^\phi-1-\phi)-e^{V_\varepsilon}(e^\phi-1-\phi).}
In one hand, from (\ref{estimate}) we deduce
\equ{\|e^{V_\varepsilon}(e^{\phi}-1-\phi)\|_*\leq C\varepsilon\|\phi\|_\infty^2.}
In the other hand
\equ{\lambda^{-2}f(\varepsilon y) e^{V_\varepsilon(y)}=\lambda^2\delta^2 e^{v_\varepsilon(x)},\quad y=\frac{x}{\varepsilon}, x\in M,}
so
\equ{\|\lambda^{-2}f(\varepsilon y)e^{V_\varepsilon}(e^{\phi}-1-\phi)\|_*\leq C\varepsilon^{-\sigma}\|\phi\|_\infty^2.}
We conclude,
\equ{\|N(\phi)\|_*\leq C\varepsilon^{-\sigma}\|\phi\|^2_\infty.}
Observe that for $\phi_1,\phi_2\in H_\vartheta$,
\equ{\|N(\phi_1)-N(\phi_2)\|_*\leq C\vartheta\varepsilon^{1-\sigma}\log\left(\frac{1}{\varepsilon}\right)\|\phi_1-\phi_2\|_\infty,}
where $C$ is independent of $\vartheta$. Hence, we have
\begin{align*}
\|A(\phi)\|_\infty &\leq C \varepsilon \log\left(\frac{1}{\varepsilon}\right)[\vartheta^2\varepsilon^{1-\sigma}\log\left(\frac{1}{\varepsilon}\right)+1],\\
\|A(\phi_1)-A(\phi_2)\|_\infty&\leq C \varepsilon^{1-\sigma} \log\left(\frac{1}{\varepsilon}\right) \|\phi_1-\phi_2\|_\infty.
\end{align*}
It follows that there exist $\varepsilon_0$, such that for all $\varepsilon<\varepsilon_0$ the operator $A$ is a contraction mapping from $H_\vartheta$ into itself, and therefore $A$ has a unique fixed point in $H_\vartheta$. This concludes the proof.
\end{proof}

With these ingredients we are now ready for the proof of our main result.

\section{Proof of Theorem \ref{main} for $n=1$}\label{sec6}
After problem (\ref{ProbNonLin}) has been solved, we find a solution to problem (\ref{ProbPhi}), and hence to the original problem, if $k=k(\varepsilon)$ is such that
\eq{c_i(k)=0,\quad i=1,2.\label{eqCi}}
Let us consider local conformal coordinates around $p$ and define $r=|y|$. We consider a smooth cut-off function $\eta(r)$ such that $\eta(r)=1$ for $r<\frac{1}{\sqrt{\varepsilon}}$, $\eta(r)=0$ for $r>\frac{2}{\sqrt{\varepsilon}}$, $|\eta'(r)|\leq C\sqrt{\varepsilon}$, $|\eta''(r)|\leq C\varepsilon$. Testing the equation
\equ{L(\phi)=N(\phi)+E+\sum_{i=1}^{2}c_i\chi Z_i,}
against $\eta Z_i$, $i=1,2,$ we find
\equ{\langle L(\phi),\eta Z_i\rangle=\int_{M_\varepsilon}[N(\phi)+E]\eta Z_i +c_i\int_{M_\varepsilon}\chi Z_i^2,\quad i=1,2.}
Therefore, we have the validity of (\ref{eqCi}) if and only if
\equ{\langle L(\phi),\eta Z_i\rangle-\int_{M_\varepsilon}[N(\phi)+E]\eta Z_i=0,\quad i=1,2.}
We recall that in the proof of Proposition (\ref{PropLin}) we obtained
\equ{|\langle \phi,L(\eta Z_i) \rangle|\leq C\sqrt\varepsilon\|\phi\|_\infty,}
thus
\equ{|\langle \phi,L(\eta Z_i) \rangle| \leq C \varepsilon^{3/2}\log \frac{1}{\varepsilon}.}
Observe that
\equ{\|N(\phi)\|_\infty\leq C\varepsilon^2\|\phi\|^2_\infty,}
so
\equ{\left|\int_{M_\varepsilon}N(\phi)\eta Z_i\right|\leq C\varepsilon\|\phi\|_\infty^2\leq C \varepsilon^3 \log^2 \frac{1}{\varepsilon}}
Let us remember that
\equ{E=-\Delta V_\varepsilon +\lambda^{-2}f(\varepsilon y)e^{V_\varepsilon}-e^{V_\varepsilon}-\varepsilon^2\alpha.}
Using (\ref{estimate}), we have
\equ{\int_{M_\varepsilon}e^{V_\varepsilon}\eta Z_i=O(\varepsilon^4).}
We also have,
\equ{\int_{M_\varepsilon}\varepsilon^2\alpha \eta Z_i=O(\varepsilon).}
Observe that
\equ{\Delta_g V_\varepsilon(y)=\varepsilon^2 \Delta_g v_\varepsilon(x), \quad y=\frac{x}{\varepsilon},x\in M,}
thus
\equ{\int_{M_\varepsilon}\Delta V_\varepsilon \eta Z_i=O(\varepsilon^2).}
Also, by change of variables we have
\equ{\int_{M_\varepsilon} f(\varepsilon y)e^{V_\varepsilon}\eta Z_i=\int_{\tilde{M}_\varepsilon}f(p+\varepsilon(y+k))e^{V_\varepsilon(y+k+p')}\eta(|y+k|)Z_i(y+k+p'),}
where $\tilde{M}_\varepsilon=M_\varepsilon-{k+p'}$. Using the fact that $p$ is a local maximum of $f$ of value $0$, we have
\equ{f(p+\varepsilon(y+k))=\varepsilon^2\langle (y+k),D^2f(p)(y+k)\rangle+O(\varepsilon^3),}
where we used the fact that $f\in C^3(M)$. Thus
\equ{\lambda^{-2}\int_{M_\varepsilon} f(\varepsilon y)e^{V_\varepsilon}\eta Z_i= I_i+II_i,}
where
\begin{align*}
I_i&=\delta^2\int_{\tilde{M}_\varepsilon}\langle (y+k),H_f(p)(y+k)\rangle e^{V_\varepsilon(y+k+p')}\eta(|y+k|)Z_i(y+k+p')\\
II_i&=\int_{\tilde{M}_\varepsilon}O(\varepsilon)e^{V_\varepsilon(y+k+p')}\eta(|y+k|)Z_i(y+k+p').
\end{align*}
Observe that $e^{V_\varepsilon(y+k+p')}\eta(|y+k|)Z_i(y+k+p')=O((1+|y|)^{-4})$, so
\equ{II_i=O(\varepsilon).}
Finally, let us compute $I_i$. In the first place, observe that $0\in \tilde{M}_\varepsilon$. Let us consider a fixed number $A_0$, such that $\mathcal{B}_1=B(0,A_0/\sqrt{\varepsilon})\subset \tilde{M}_\varepsilon\cap\mbox{supp}(\eta(\cdot+k)):=\mathcal{B}$ and $\eta(\cdot+k)=1$ in $\mathcal{B}_1$. We have the decomposition $\mathcal{B}=\mathcal{B}_1+\mathcal{B}_2$, where $\mathcal{B}_2=\tilde{\Omega}_\varepsilon\cap\mbox{supp}(\eta(\cdot+k))\setminus \mathcal{B}_1$. Also, observe that
\equ{Z_i(y+k+p')=C_0\frac{y_i}{1+|y|},\quad i=1,2,}
where $C_0$ is a fixed constant independent of $\varepsilon$. We have the following computation
\equ{\langle (y+k),D^2f(p)(y+k)\rangle= f_{11}(p)(y_1+k_1)^2+2f_{12}(p)(y_1+k_1)(y_2+k_2)+f_{22}(p)(y_2+k_2)^2,}
where $f_{11}(p)=\frac{\partial^2 f}{\partial y_1^2}(p)$, $f_{22}(p)=\frac{\partial^2 f}{\partial y_2^2}(p)$ and $f_{12}(p)=f_{21}(p)=\frac{\partial^2 f}{\partial y_1\partial y_2}(p)$. We recall that
\eq{e^{V_\varepsilon(y+k+p')}=\frac{H_0}{(1+|y|^2)^2}(1+C\sqrt{\varepsilon}+O(\varepsilon)),\label{taylor2}}
in the region $\tilde{\Omega}_\varepsilon\cap\mbox{supp}(\eta(\cdot+k))$. We define $t(y)=e^{V_\varepsilon(y+k+p')}\eta(|y+k|)Z_i(y+k+p')$. We have
\begin{align*}
\int_{\mathcal B}f_{11}(p)(y_1+k_1)^2 t(y)&=\int_{\mathcal B_1}f_{11}(p)(y_1+k_1)^2 t(y)+\int_{\mathcal B_2}f_{11}(p)(y_1+k_1)^2 t(y)\\
&=2k_1f_{11}(p)\int_{\mathcal{B}_1}C_0\frac{y_1^2}{1+|y|}\frac{H_0}{(1+|y|^2)^2}+O(\varepsilon).
\end{align*}
In order to get the previous result, we used the fact that
\equ{\int_{\mathcal B_1}\frac{y_1}{1+|y|}\frac{dy}{(1+|y|^2)^2}=\int_{\mathcal B_1}\frac{y_1^3}{1+|y|}\frac{dy}{(1+|y|^2)^2}=0,}
and the expansion (\ref{taylor2}). We also have
\equ{
\int_{\mathcal B_1+\mathcal B_2}2f_{12}(p)(y_1+k_1)(y_2+k_2) t(y)=2k_2f_{12}(p)\int_{\mathcal{B}_1}C_0\frac{y_1^2}{1+|y|}\frac{H_0}{(1+|y|^2)^2}+O(\varepsilon),}
where we used the fact that
\equ{\int_{\mathcal B_1}\frac{y_1y_2}{1+|y|}\frac{1}{(1+|y|^2)^2}=0,}
and also the expansion (\ref{taylor2}). Finally, we have
\equ{\int_{\mathcal B_1+\mathcal B_2}f_{22}(p)(y_2+k_2)^2 t(y)=O(\varepsilon),}
where we used the fact that
\equ{\int_{\mathcal B_1}\frac{y_1y_2^2}{1+|y|}\frac{1}{(1+|y|^2)^2}=0,}
and also the expansion (\ref{taylor2}). From the above computations we conclude that
\equ{I_1=2\delta^2Ik_1f_{11}(p)+2\delta^2Ik_2f_{12}(p)+O(\varepsilon),}
where
\equ{I=\int_{\mathcal{B}_1}C_0\frac{y_1^2}{1+|y|}\frac{H_0}{(1+|y|^2)^2} >0 .}
Similar computations yield
\equ{I_2=2\delta^2Ik_1f_{12}(p)+2\delta^2Ik_2f_{22}(p)+O(\varepsilon).}

Summarizing, we have the system
\eq{\delta^2D^2f(p)k=\varepsilon b(k),\label{sys}}
where $b$ is a continuous function of $k$ of size $O(1)$. Since $p$ is a non-degenerate critical point of $f$, we know that $D^2f(p)$ is invertible. A simple degree theoretical argument, yields that system (\ref{sys}) has a solution $k=O(\lambda\delta^{-1})$. We thus obtain $c_1(k)=c_2(k)=0$, and we have found a solution of the original problem. The proof for the case $k=1$ is thus concluded. \qed
\section{Proof of Theorem \ref{main} for  general $n$ }
In this section we will detail the main changes in the proof of our main result, in the case of multiple bubbling. \newline

Let $p_1,\ldots,p_n$ be points such that $f(p_j)=0$ and $D^2f(p_j)$ is positive definite for each $j$. We consider the singular problem
\eq{\Delta_g G-fe^G+8\pi\sum_{j=1}^k \delta_{p_j}+\alpha=0, \quad \mbox{in }M,\label{multi}}
where $\delta_{p}$ designates the Dirac mass at the point $p$.
A first remark we make is that the proof of Lemma \ref{lema2} applies with no changes (except some additional notation) to find the result of Lemma \ref{lema1}. Indeed, the core of the proof is the local asymptotic analysis around each point $p_j$.

\medskip

We define the first approximation in the original variable as
$$U_\varepsilon=\sum_{j=1}^{n}\eta_ju_\varepsilon^j+\left(1-\sum_{j=1}^{n}\eta_j\right)G,$$
where $\eta_j$ is defined around $p_j$ as in Section \ref{SecAprox} and, in local conformal coordinates around $p_j$, $u_\varepsilon^j(x)=w_{\delta_j}(|x-k_j|)+\tilde{F}_j(|x|)$, for parameters $k_j\in \R^2$. We make the following choice of the parameters $\delta_j$
\equ{\log 8\delta_i^2=-2\log\left(\frac{1}{\sqrt 2}\log \frac{1}{\lambda}\right)+\mathcal{H}(p_i).}
We also define the first approximation in the expanded variable around each $p_j$ by
\equ{V_{\ve_j }(y)=U_{\ve}(\ve_j y)+4\log \lambda + 2 \log \delta_j,\quad y\in M_{\ve_j}}
where $\ve_j=\lambda \delta_j$ and $M_{\ve_j}=\ve_j^{-1}M$.

\medskip
We look for a solution of problem (\ref{problem}) of the form $u(y)=U_\varepsilon(x)+\phi(x)$, where $\phi$ represent a lower order correction. By simplicity, we denote also by $\phi$ the small correction in the expanded variable around each $p_j$. In terms of $\phi$, the expanded problem around $p_j$
\equ{\Delta_g v-\lambda^{-2}f(\ve_j y)e^v+e^v+\ve_j^2\alpha=0,\quad y\in M_{\ve_j}}
reads
\equ{ L_j(\phi)=N_j(\phi)+E_j,\quad\mbox{in }M_{\varepsilon_j}, }
where
\begin{align*}
L_j(\phi)\colonequals&\Delta_g \phi -\lambda^{-2}f(\varepsilon_j y)e^{V_{\varepsilon_j}}\phi+e^{V_{\varepsilon_j}}\phi,\\
N_j(\phi)\colonequals&\lambda^{-2}f(\varepsilon_j y)e^{V_{\varepsilon_j}}(e^\phi-1-\phi)-e^{V_{\varepsilon_j}}(e^\phi-1-\phi),\\
E_j\colonequals&-(\Delta_g V_{\varepsilon_j}-\lambda^{-2}f({\varepsilon_j} y)e^{V_{\varepsilon_j}}+e^{V_{\varepsilon_j}}+{\varepsilon_j}^2\alpha).
\end{align*}

\medskip
Next we consider the linearized problem around our first approximation $U_\ve$.
Given $h$ of class $C^{0,\beta}(M)$, which by simplicity we still denote by $h$ in the expanded variable around each $p_j$, we consider the linear problem of finding a function $\phi$ such that for certain scalars $c_i^j,\,i=1,2;\ j=1,\ldots, n  $, one has
\eq{\left\lbrace
\begin{array}{rcll}
L_j(\phi)&=&h+\sum_{i=1}^{2}\sum_{j=1}^n c_i^j\chi_j Z_{ij}&\mbox{in }M_{\varepsilon_j},\\
\displaystyle \int_{M_{\varepsilon_j}} \chi_j Z_{ij} \phi&=&0&\mbox{for all } i,j.
\end{array}\right.\label{ProbPhi5}}
Here the definitions of $Z_{ij}$ and $\chi_j$ are the same as before for $Z_i$ and $\chi$, with the dependence of the point $p_j$ emphasized.

\medskip
 To solve this problem we consider now the norm
\eq{\|h\|_{*} =  \sum_{j=1}^n \|h\|_{*,p_j}. \label{norma1} }
where $\|h\|_{*,p_j}$ is defined accordingly with \eqref{norma}.
With exactly the same proof as in the case $n=1$, we find the unique bounded solvability of Problem
\ref{ProbPhi5} for all small $\ve=\max \ve_i$ by $\phi=T(h)$, so that
\eq{\|T(h)\|_\infty\leq C\log\left(\frac{1}{\varepsilon}\right)\|h\|_{*}.}
Then we argue as in the proof of
Lemma \ref{lemita} to obtain existence and uniqueness of a small solution $\phi$ of the projected nonlinear problem

\equ{\left\lbrace
\begin{array}{rcll}
L_j(\phi)&=&N_j(\phi)+E_j+
\sum_{i=1}^{2}\sum_{j=1}^n c_i^j\chi_j Z_{ij}&\mbox{in }M_{\varepsilon_j},\\
\displaystyle \int_{M_{\varepsilon_j}} \chi_j Z_{ij} \phi&=&0&\mbox{for all } i,j.
\end{array}\right.\label{ProbNonLin2}}
with
\equ{\|\phi\|_\infty \leq C\varepsilon \log \frac{1}{\varepsilon}.}

After this, we proceed as in Section \ref{sec6} to choose the parameters $k_j$ in such a way
that $c_i^j=0$ for all $i,j$.
Summarizing, we have the system
\eq{D^2f(p_j)k_j=\varepsilon_i \delta_i^{-2} b_j(k_1,\ldots, k_n),\label{sys1}}
which can be solved by the same degree-theoretical argument employed before.
The proof is concluded. \qed

\begin{center}
\bfseries{Acknowledgement}
\end{center}
The authors have been supported by grants Fondecyt 110181 and Fondo Basal CMM-Chile.

\end{document}